\documentclass[twoside]{article}

\usepackage[accepted]{aistats2020}
\usepackage[round]{natbib}

\usepackage{hyperref}
\usepackage{multicol}
\usepackage{booktabs}       
\usepackage{amsfonts}       
\usepackage{nicefrac}       
\usepackage{microtype}      
\usepackage{xcolor}
\usepackage{graphicx}
\usepackage{amsmath,amsthm,amssymb,amsfonts}
\usepackage{mathrsfs}
\usepackage{footnote}
\usepackage{enumerate}
\usepackage[linesnumbered,ruled,vlined]{algorithm2e}
\usepackage[noend]{algpseudocode}
\usepackage{epstopdf}
\usepackage{multicol,lipsum}
\newtheorem{theo}{Theorem}[section]
\newtheorem{theorem}[theo]{Theorem}
\newtheorem{proposition}[theo]{Proposition}
\newtheorem{lemma}[theo]{Lemma}
\newtheorem{definition}[theo]{Definition}

\newtheorem{remark}[theo]{Remark}
\newtheorem{assumption}[theo]{Assumption}
\newtheorem{corollary}[theo]{Corollary}

\usepackage{makecell}
\usepackage[titletoc]{appendix}

\begin{document}
\runningtitle{AsyncQVI for Discounted Markov Decision Processes with Near-Optimal Sample Complexity}
%

%

\twocolumn[

\aistatstitle{AsyncQVI: Asynchronous-Parallel Q-Value Iteration for Discounted Markov Decision Processes with Near-Optimal Sample Complexity}

\aistatsauthor{ Yibo Zeng \And Fei Feng \And Wotao Yin }

\aistatsaddress{Fudan University\\Columbia University \And University of California, \\Los Angeles \And University of California, \\Los Angeles}



]

\begin{abstract}
  In this paper, we propose AsyncQVI, an asynchronous-parallel Q-value iteration for discounted Markov decision processes whose transition and reward can only be sampled through a generative model. Given such a problem with $|\mathcal{S}|$ states, $|\mathcal{A}|$ actions, and a discounted factor $\gamma\in(0,1)$, AsyncQVI uses memory of size $\mathcal{O}(|\mathcal{S}|)$ and returns an $\varepsilon$-optimal policy with probability at least $1-\delta$ using $$\tilde{\mathcal{O}}\big(\frac{|\mathcal{S}||\mathcal{A}|}{(1-\gamma)^5\varepsilon^2}\log(\frac{1}{\delta})\big)$$  samples.\footnote{We use $\tilde{\mathcal{O}}$ to omit polylogarithmic factors, i.e., $\tilde{\mathcal{O}}(f)=\mathcal{O}(f\cdot (\log f)^{\mathcal{O}(1)})$.\label{tilde_O}} AsyncQVI is also the first asynchronous-parallel algorithm for discounted Markov decision processes that has a sample complexity, which nearly matches the theoretical lower bound. The relatively low memory footprint and parallel ability make AsyncQVI suitable for large-scale applications. In numerical tests, we compare AsyncQVI with four sample-based value iteration methods. The results show that our algorithm is highly efficient and achieves linear parallel speedup.
\end{abstract}

\section{Introduction}
Markov Decision Processes (MDPs) are a fundamental model to encapsulate sequential decision making under uncertainty. They have been indepthly studied and successfully applied to many fields, especially Reinforcement Learning (RL). As a rapidly developing area of artificial intelligence, RL is being flourishingly combined with deep neural network~\citep{Mnih2015Human,mnih2016asynchronous,li2017deep} and used in many domains including games~\citep{Mnih2015Human, silver2016mastering}, robotics~\citep{kober2013reinforcement}, natural language processing~\citep{young2018recent}, finance~\citep{deng2016deep}, healthcare~\citep{kosorok2015adaptive} and so on. With the advent of big-data applications, computational costs have increased significantly. Therefore, parallel computing techniques have been applied to reduce RL solving time~\citep{grounds2008parallel, nair2015massively}. 
Recently, asynchronous (async) parallel algorithms have been widely researched in RL and gained empirical success~\citep{mnih2016asynchronous, babaeizadeh2016reinforcement,  gu2017deep, stooke2018accelerated, zhang2019asynchronous}. Compared to synchronous (sync) parallel algorithms, where the agents must wait for the slowest agent to finish its task before they can all proceed to the next one, async-parallel algorithms allow agents to run continuously with little idling. Hence, async-parallel algorithms complete more tasks than their synchronous counterparts (though information delays and inconsistencies may negatively affect the task quality). Async-parallel algorithms have other advantages~\citep{bertsekas1991some}: the system is more tolerant of computing faults and communications glitches; it is also easy to incorporate new agents. 

In contrast to promising empirical results in async-parallel RL, its theoretical property has not been fully understood. In this paper, we are trying to mitigate the gap between theory and practice. Specifically, we will asynchronous-parallelly solve Discounted Infinite-Horizon Markov Decision Processes (DMDPs) which is not fully known in advance. A DMDP is described by a tuple $(\mathcal{S}, \mathcal{A}, \mathrm{P}, \mathrm{r}, \gamma)$, where $\mathcal{S}$ is a finite state space, $\mathcal{A}$ is a finite action space, $\mathrm{P}$ contains the transition probabilities, $\mathrm{r}$ is the collection of instant rewards, and $\gamma\in(0,1)$ is a discounted factor. At time step $t$, the controller or the decision maker observes a state $s_t \in \mathcal{S}$ and selects an action $a_t \in \mathcal{A}$ according to a policy $\pi$, where $\pi$ maps a state to an action. The action leads the environment to a next state $s_{t+1}$ with probability $p_{s_ts_{t+1}}^{a_t}$. Meanwhile, the controller receives an instant reward $r_{s_ts_{t+1}}^{a_t}$. Here, $r^{a_t}_{s_ts_{t+1}}$ is a deterministic value given the transitional instance $(s_t,a_t,s_{t+1})$. If only $s_t$ and $a_t$ are specified, $r^{a_t}_{s_t}$ is a random variable and $r^{a_t}_{s_t}=r^{a_t}_{s_ts_{t+1}}$ with probability $p^{a_t}_{s_ts_{t+1}}$. Given a policy $\pi: \mathcal{S} \rightarrow \mathcal{A}$,
we denote $\mathbf{v}^{\pi}\in\mathbb{R}^{|\mathcal{S}|}$ the state-value vector of $\pi$. Specifically,
\begin{equation*}
 \resizebox{1\hsize}{!}{$\mathbf{v}^{\pi} := [v^\pi_{1}, v^\pi_{2}, \cdots, v^\pi_{|\mathcal{S}|}]^\top,~v^{\pi}_i:= \mathbb{E}^\pi\big[\sum_{t=0}^\infty \gamma^{t} r_{s_t s_{t+1}}^{a_t}|s_0=i\big],$}
\end{equation*}
where the expectation 
is taken over the trajectory $(s_0,a_0,s_1,a_1,\ldots,s_t,a_t\dots)$ following $\pi$, i.e. $a_t=\pi_{s_t}$.
The objective is to seek for an optimal policy $\pi^*$ such that $\mathbf{v}^{\pi}$ is maximized component-wisely.

In our setting, P and r are unknown, which is also
the case for RL. Thus, an optimal policy cannot be obtained through dynamic programming approach but learned from transitional samples.
Depending on the applications, one can have access to either trajectories samples or a \emph{generative model}.
Specifically, given any state-action pair $(i,a)$, a {generative model} returns a next state $j$ with probability $p_{ij}^a$ and the instant reward $r_{ij}^a$. One can repeatedly call it with the same input $(i,a)$. Although the generative model is a stronger assumption than trajectories samples, it is natural and practical in many cases. Our algorithm must access a {generative model}; as a benefit, the algorithm requires only $\mathcal{O}(|\mathcal{S}|)$ memory and achieves a nearly optimal sample complexity.

We use notation $\mathbf{p}_{i}^{a}:=[p_{i1}^a, p_{i2}^a, \cdots,$ $p_{i|\mathcal{S}|}^a]^\top$ and $\bar{r}_i^a:=\sum_{j\in\mathcal{S}} p_{ij}^a r_{ij}^a$ and assume, without loss of generality, $r_{ij}^a\in [0,1]$, $\forall~ i,j\in\mathcal{S}, a\in\mathcal{A}$. We let $\mathbf{v}^*$ denote the optimal value vector associated with an optimal policy $\pi^*$. A policy $\pi$ is $\varepsilon$-optimal if  $\Vert \mathbf{v}^*-\mathbf{v}^{\pi} \Vert_\infty\leq \varepsilon$.

In this paper, we propose the algorithm Asynchronous-Parallel Q-Value Iteration (AsyncQVI), the first async-parallel RL algorithm that
has a sample complexity result. AsyncQVI returns an $\varepsilon$-optimal policy with probability at least $1-\delta$ using
$$
\tilde{O}\Big(\frac{|\mathcal{S}||\mathcal{A}|}{(1-\gamma)^5\varepsilon^2}\log(\frac{1}{\delta})\Big) $$
samples\textsuperscript{\ref{tilde_O}},
provided that each coordinate is updated at least once within $\mathcal{O}(|\mathcal{S}||\mathcal{A}|)$ time and the async delay is bounded by $\mathcal{O}(|\mathcal{S}||\mathcal{A}|)$. ~\citet{sidford2018near} established the lower bound on the sample complexity of any DMDP with a generative model as
$$
  \Omega\Big(\frac{|\mathcal{S}||\mathcal{A}|}{(1-\gamma)^3\varepsilon^2}
  \log(\frac{1}{\delta})\Big)
$$
for finding an $\varepsilon$-optimal policy $\pi$ with probability at least $1-\delta$. Therefore, our result nearly matches the lower bound up to $(1-\gamma)^2$ and logarithmic factors.
Besides, AsyncQVI requires only $\mathcal{O}(|\mathcal{S}|)$ memory, which is minimal possible (without using dimension reduction) to store $\pi$.

With a near-optimal sample complexity, the minimal memory requirement, and async-parallel implementation, AsyncQVI is a competitive RL algorithm.


\paragraph{Notation} We write a scalar in  \emph{italic} type, a vector or a matrix in \textbf{boldface}, and their components with subscripts. For example, $\mathbf{v}$ and $v_i$ are a vector and its $i$th component, respectively.

\section{Related Works}
AsyncQVI is not the first attempt to solve DMDP problems with asynchronous parallel. As early as in \citet{bertsekas1989parallel}, the authors proposed async-parallel dynamic programming methods. They established and analyzed fundamental asynchronous models, which are characterized by coordinate update and asynchronous delay. This seminal work inspires the later study of async-parallel algorithms for DMDPs that are not fully known beforehand and can only be accessed by samples.

\citet{Tsitsiklis1994} adapted Q-learning to async setting and provided the convergence guarantee. However, although several works have established sample complexity results for single-threaded cases \citep{kearns2002sparse, even2003learning, azar2011speedy, azar2013minimax, kalathil2014empirical, sidford2018near, sidford2018variance, agarwal2019optimality}, there have been no such results for async-parallel algorithms. Moreover, considering the latent huge cost of taking samples, an explicit complexity result is more and more concerned and is an important algorithm comparison reference.

One may notice that to achieve promising complexity results, several works adopt the \emph{generative model}, e.g.,~\citet{kearns2002sparse, azar2011speedy, azar2013minimax, kalathil2014empirical, sidford2018near, sidford2018variance}. This model is proposed by~\citet{kearns2002sparse}. It is indeed a sample oracle which takes any state-action pair $(i, a)$ as input and returns a next state $j$ with probability $p_{ij}^a$ and the corresponding instant reward $r_{ij}^a$. Our algorithm is also built under the generative model and we develop the first async-parallel algorithm that has an explicit sample complexity.

\begin{table*}
\caption{Related Async-parallel Methods For DMDPs.}\label{asyn_compare}
\vspace{-1em}
\begin{center}
\begin{tabular}{cccccc}
  \textbf{Algorithms}
  & \textbf{Assumption}
  & \textbf{Async Delay} & \makecell{\textbf{Sample} \\\textbf{Complexity}}
  & \textbf{Memory}
  & \textbf{References}\\
  \hline\\
  \makecell{Totally Async\\ QVI}  & Fully known DMDP
  & Unbounded\protect\footnotemark 
  & N/A &$\mathcal{O}(|\mathcal{S}||\mathcal{A}|)$
  &\makecell{Bertsekas and\\ Tsitsiklis (1989)}\\

  \makecell{Partially Async\\ QVI} & Fully known DMDP
  & Uniformly Bounded
  & N/A
  &$\mathcal{O}(|\mathcal{S}||\mathcal{A}|)$
  & \makecell{Bertsekas and\\ Tsitsiklis (1989)}\\

  Async Q-learning  & Trajectory samples
  & Unbounded\textsuperscript{\ref{unbounded_delay}}
  & $-$ 
  &$\mathcal{O}(|\mathcal{S}||\mathcal{A}|)$
  &\citet{Tsitsiklis1994}\\

  AsyncQVI  & Generative model
  & Uniformly Bounded & $\surd$ 
  &$\mathcal{O}(|\mathcal{S}|)$
  &This Paper \\
\end{tabular}
\end{center}
\end{table*}
\begin{table*}
\caption{Related Algorithms For DMDP With A Generative Model.}
\label{GM_compare}
\begin{center}
\begin{tabular}{ccccc}
  \textbf{Algorithms} & \textbf{Async} &\textbf{Sample Complexity}  & \textbf{Memory}  & \textbf{References}\\
  \hline\\
  Variance-Reduced VI
  &$\times$
  & $\tilde{O}\big(\frac{|\mathcal{S}||\mathcal{A}|}{(1-\gamma)^4\varepsilon^2}\log(\frac{1}{\delta})\big)$
  &$\mathcal{O}(|\mathcal{S}||\mathcal{A}|)$
  &\citet{sidford2018variance} \\

  Variance-Reduced QVI
  &$\times$
  & $\tilde{O}\big(\frac{|\mathcal{S}||\mathcal{A}|}{(1-\gamma)^3\varepsilon^2}\log(\frac{1}{\delta})\big)$
  &$\mathcal{O}(|\mathcal{S}||\mathcal{A}|)$
  &\citet{sidford2018near} \\

  AsyncQVI
  &$\surd$
  & $\tilde{O}\big(\frac{|\mathcal{S}||\mathcal{A}|}{(1-\gamma)^5\varepsilon^2}\log(\frac{1}{\delta})\big)$
  &$\mathcal{O}(|\mathcal{S}|)$
  &This Paper \\
\end{tabular}
\end{center}
\end{table*}

We list related async-parallel methods for DMDPs in Table~\ref{asyn_compare} and the generative model methods in Table~\ref{GM_compare}. Note that some papers~\citep{even2003learning, azar2011speedy, kalathil2014empirical} use the word ``asynchronous'' for single-threaded coordinate update methods. In constrast, our algorithm is not only multi-threaded, but also allows stale information and async delay.
Further, the lower sample complexities achieved by~\citet{sidford2018near, sidford2018variance} rely on the \emph{variance reduction} technique, which requires periodic synchronization and $\mathcal{O}(|\mathcal{S}||\mathcal{A}|)$ memory footprint to update and store a basis, say ${\mathbf{p}_i^a}^\top\mathbf{v}_0$, $\forall$ $i\in\mathcal{S}, a\in\mathcal{A}$. In order to take advantage of fully async-parallel structure and achieve the minimal memory complexity $\mathcal{O}(|\mathcal{S}|)$, we do not implement variance reduction and therefore, obtain a slightly higher sample complexity.



The last thing to mention is that there are some other nice async-parallel works about fixed point problems in a Hilbert space, e.g. \citet{peng2016arock,HannahYin2016_UnboundedDelay}, while our algorithm is based on a contraction with respect to the $\ell_\infty$ norm.
\footnotetext{Under the assumption: $\forall i,j$, $\lim_{t\rightarrow\infty}\tau_j^i(t)=\infty$ holds with probability 1.~\label{unbounded_delay}}
\section{Preliminaries}\label{preliminary}

In this section, we review several key results on Q-value iteration and async-parallel algorithms. 

\subsection{Q-value Iteration}\label{pre_Q}
Given a DMDP $(\mathcal{S},\mathcal{A},\mathrm{P},\mathrm{r},\gamma)$ and
a policy $\pi$, we define the action-value vector $\mathbf{Q}^{\pi}$ with entries
\begin{equation*}
    Q^{\pi}_{i,a} = \mathbb{E}^\pi[\sum_{t=0}^\infty \gamma^{t} r_{s_t s_{t+1}}^{a_t}\big\vert~ s_0 = i, a_0 = a].
\end{equation*}
For an optimal policy $\pi^*$, we let $\mathbf{Q}^*$ denote the corresponding optimal action-value vector.
From $\mathbf{Q}^*$, we can obtain $\forall~i\in \mathcal{S}$,
$\pi^*_i = \arg\max_{a} Q^{*}_{i,a}$, $v^{*}_{i} = \max_{a} Q^{*}_{i,a}$. Hence, to derive an optimal policy $\pi^*$, it suffices to compute the optimal action-value vector $\mathbf{Q}^*$. 
To reach this end, we first define an operator $T: \mathbb{R}^{|\mathcal{S}||\mathcal{A}|}\rightarrow \mathbb{R}^{|\mathcal{S}||\mathcal{A}|}$ as
\begin{equation}~\label{Q_value_operator}
    \resizebox{1\hsize}{!}{$[T\mathbf{Q}]_{i,a}
=   \underbrace{\sum_{j\in\mathcal{S}} p_{ij}^{a} r_{ij}^a}_{\text{expected instant reward}}
    +\underbrace{\gamma\sum_{j\in\mathcal{S}} p_{ij}^{a}\max_{a^\prime}{Q}_{j,a^\prime}}_{\text{expected discounted future reward}},$}
\end{equation}
where $\mathbf{Q}\in \mathbb{R}^{|\mathcal{S}||\mathcal{A}|}$ and $[T\mathbf{Q}]_{i,a}$ is the $((i-1)\times|\mathcal{A}|+a)$th component of $T\mathbf{Q}$ with $1\leq i\leq|\mathcal{S}|,1\leq a\leq |\mathcal{A}|$. Actually, $T$ is the well-known Bellman operator. It is an $\gamma-$contraction under $\ell_\infty$ norm and $\mathbf{Q}^*$ is the unique fixed point (see e.g.~\citep{puterman2014markov}). Therefore, one can apply fixed-point iterations of $T$ to recover $\mathbf{Q}^*$. Next, we introduce the async-parallel coordinate update fashion of fixed-point iterations.

\subsection{Asynchronous-Parallel Coordinate Updates}\label{pre_async}

\begin{algorithm}[h]
\caption{Asynchronous-Parallel Coordinate Updates}\label{Async_Alg}
\textbf{Shared variables: $\mathbf{x}^0$, $L>0$, $t\gets 0$};

\textbf{Private variable: $\hat{\mathbf{x}}$};

\While{ $t < L$, every agent asynchronously \do}
{   select $i\in \{1, 2,\cdots, n\}$ according to some criterion;

    read (required) shared variable to local memory $\hat{\mathbf{x}}\gets\mathbf{x}$;

    perform an update $x_{i} \gets G_{i}(\hat{\mathbf{x}})$;\nllabel{Alg1_update}

    increment the global counter $t\gets t+1$;
}
\end{algorithm}

Given an $\ell_\infty$ $\gamma-$ contraction $G:\mathbb{R}^n \rightarrow \mathbb{R}^n$, the fixed-point iteration
$\mathbf{x}(t+1) = G(\mathbf{x}(t))$, $t \geq 0$
converges linearly. Rewriting $G\mathbf{x}$ as $(G_1\mathbf{x}, \ldots, G_n\mathbf{x})$, we call
\begin{align}\label{coordinateupdate}
  x_i(t+1) = \begin{cases}
      G_{i}(\mathbf{x}(t)),& t\in \mathscr{T}^{i}; \\
      x_{i}(t),   & t\notin  \mathscr{T}^{i},
    \end{cases}
\end{align}
the coordinate update of $G\mathbf{x}$,
where $x_i(t)$ is the $i$th coordinate of $\mathbf{x}$ at iteration $t$ and
$$\mathscr{T}^{i} :=\{t\ge 0: \text{coordinate}~i~\text{is updated at iteration}~t\}$$ is the set of iterations at which $x_i$ is updated.

We use a set of computing agents to perform coordinate update \eqref{coordinateupdate} in an async-parallel fashion. 
Unlike the typical parallel implementation where all the agents must wait for the slowest one to finish an update, async algorithms allow each agent to use the (possibly stale) information it has and complete more iterations within the same period of time, which is preferable for cases where the computing capacity is highly heterogeneous or the workload is far from balanced. See more discussions in \citet{hannah2017more}.

We summarize a shared-memory async-parallel coordinate-update framework in Algorithm~\ref{Async_Alg}, where each agent first chooses one coordinate to update, then reads necessary information from global memory to the local cache, and finally updates its computed result to the shared memory.

%

By Line~\ref{Alg1_update} in Algorithm~\ref{Async_Alg}, the $t$th update 
can be written as
\begin{equation}\label{Alg1_Update1}
  x_{i}(t+1) =
  \begin{cases}
   G_i(\hat{\mathbf{x}}(t)),
                    & \hbox{$t\in \mathscr{T}^{i}$;} \\
      x_{i}(t),   & \hbox{$t\notin  \mathscr{T}^{i}$.}
  \end{cases}
\end{equation}
Here, $\hat{\mathbf{x}}(t):=[x_{1}(\tau_{1}(t)), \ldots,x_{n}(\tau_{n}(t))]^\top$ represents the possibly stale information, where
$x_j(\tau_j(t))$ is the most recent version of $x_j$
available at time $t$ that is used to compute $x_i(t+1)$. We have that $0\leq \tau_j(t) \leq t$. The difference $t - \tau_j(t)$ is called the \emph{delay}.
%
In this paper, we assume \emph{partial asynchronism}~\citep{bertsekas1989parallel}:
\begin{assumption}[Partial Asynchronism\footnote{Assumption 1.1 in \citet[Section 7.1]{bertsekas1989parallel} uses $B$ for both $B_1$ and $B_2$. Because $B_1$ and $B_2$ are different in practice, we keep them separate to derive a tighter bound. Further, we have dropped assumption (c) there to make our algorithm easier to implement.}]\label{assump}
For the async-parallel algorithm, there exists two positive integers $B_1$, $B_2$ (asynchronism measure) such that:
\begin{enumerate}[(a)]
  \item For all $i$ and all $t\geq 0$, at least one of the elements of the set $\{t, t+1, \ldots, t+B_1-1\}$ belongs to $\mathscr{T}^{i}$;
  \item $t-B_2 < \tau_{j}(t)\leq t$ holds for all $j$ and all $t \geq 0$.
\end{enumerate}
\end{assumption}
\begin{algorithm}[htbp!]
\caption{AsyncQVI: Asynchronous-Parallel Q-value Iteration}
\label{AsyncQVI}
\KwIn{ $\varepsilon\in(0, (1-\gamma)^{-1})$,  $\delta\in(0,1)$, $L$, $K$;}
\textbf{Shared variables}: $\mathbf{v}\gets\mathbf{0}$, $\pi \gets \mathbf{0}$, $t \gets 0$;

\textbf{Private variables}: $\hat{\mathbf{v}}, r, S, q$;

\While{$t<L$, every agent asynchronously \do}
{
    select state $i_t\in\mathcal{S}$ and action $a_t\in\mathcal{A}$\;

    copy shared variable to local memory $\hat{\mathbf{v}}\gets \mathbf{v}$;\nllabel{alg_copy}

    call $\texttt{GM}(s_t, a_t)$ $K$ times and collect samples $s^\prime_1, \ldots, s^\prime_K$ and ${r_1, \ldots, r_K}$\nllabel{alg_sampling};

    $q\gets \frac{1}{K}\sum_{k=1}^K r_k +\gamma \frac{1}{K}\sum_{k=1}^K \hat{v}_{s^\prime_k}-\frac{(1-\gamma)\varepsilon}{4}$\nllabel{alg_q}\;

    \If{$q>v_{i_t}$}{\nllabel{alg_if}
    \textbf{mutex lock};\\
    $v_{i_t} \gets q$\nllabel{alg_v}, $\pi_{i_t} \gets a_t$;\nllabel{alg_pi}\\
    \textbf{mutex unlock};\nllabel{alg_unlock}}
    increment the global counter $t\gets t+1$\;
}
\Return{$\mathbf{\pi}$}\\
\end{algorithm}
\noindent Assumption~\ref{assump} (a) ensures that the time interval between consecutive updates to each coordinate is uniformly bounded by $B_1$ and (b) ensures that the communication delays are uniformly bounded by $B_2$. Note that when $B_1=B_2=1$, the algorithm becomes synchronous. 
Convergence under this assumption was established in~\citet{feyzmahdavian2014convergence}.
\begin{proposition}\label{Cov_Rate2}\emph{\citep[Theorem~2]{feyzmahdavian2014convergence}}
  Consider the iterations Eq.~\eqref{Alg1_Update1} under Assumption~\ref{assump}. Suppose that $G$ is $\gamma$-contractive under $\ell_\infty$ norm and $\mathbf{x}^*$ is the fixed point of $G$. Then
  $\Vert\mathbf{x}(t) - \mathbf{x}^*\Vert_\infty
    \leq \Vert\mathbf{x}(0) - \mathbf{x}^* \Vert_\infty \rho^{t-2B_1}$
  for all  $t\geq B_1$, where $\rho = \gamma^{\frac{1}{B_1+B_2-1}}$.
\end{proposition}

In many DMDP and RL problems, the transition probabilities $\mathrm{P}$ are sparse. So for any state-action pair $(i,a)$, the possible next states form a tiny subset of $\mathcal{S}$. Hence, to apply async-parallel coordinate updates to Eq. $\eqref{Q_value_operator}$, very few components are required and we only need to bound async dealy over a smaller subset.
Therefore, we usually have $B_2 \ll B_1$, where $B_1\geq|\mathcal{S}||\mathcal{A}|$. Hence, the convergence rate $\gamma^{\frac{1}{B_1+B_2-1}}$ we obtain is significantly better than $\gamma^{\frac{1}{2(B_1\vee B_2)-1}}$ from~\citep[Theorem.~2]{feyzmahdavian2014convergence}; the proof is deferred to Appendix A.

\begin{remark}
[Total Asynchronism] Here we do not adopt the total asynchronism notion~\citep[Section 6.1]{bertsekas1989parallel}. To start with, one cannot derive convergence rate results under total asynchronism since it allows arbitrarily long delays and no improvement can be said for finite iterations. On the contrary, partial asynchronism can avoid this case and be practically enforced~\citep[Section 7.1]{bertsekas1989parallel}. 
\end{remark}

\section{AsyncQVI: Asynchronous-Parallel Q-value Iteration}\label{main}

In this section, we present AsyncQVI and its convergence analysis.

AsyncQVI (Algorithm~\ref{AsyncQVI}) is an asynchronous stochastic version of Eq.~\eqref{Q_value_operator}. To develop AsyncQVI, we first apply the asynchronous framework (Algorithm~\ref{Async_Alg}) to Eq.~\eqref{Q_value_operator}, obtaining
\begin{equation}~\label{acu_Q}
    \resizebox{1\hsize}{!}{$Q_{i,a}(t+1)
=\begin{cases}
\sum\limits_j p_{ij}^{a} r_{ij}^a
    +\gamma\sum\limits_j p_{ij}^{a}\max\limits_{a^\prime}{\hat{Q}}_{j,a^\prime}(t), & \hbox{$t\in \mathscr{T}^{i,a}$;}\\
    Q_{i,a}(t), &\hbox{$t\notin \mathscr{T}^{i,a}$.}
    \end{cases}$}
\end{equation}
Since there is no knowledge of the transition probability, we approximate the expectations $\sum_j p_{ij}^{a}\cdot$ by random sampling (Lines~\ref{alg_sampling} and~\ref{alg_q}, Algorithm~\ref{AsyncQVI}). This is done by accessing a generative model $\texttt{GM}$, which takes a state-action pair $(i,a)$ as input and returns a next state $j$ with probability $p^{a}_{ij}$ and the corresponding instant reward $r^a_{ij}$. So instead of \eqref{acu_Q}, we substitute
$\sum_j p_{ij}^{a} r_{ij}^a$ and $\sum_j p_{ij}^{a}\max_{a^\prime}{\hat{Q}}_{j,a^\prime}(t)$ by their empirical means, i.e., $r(t):=\frac{1}{K}\sum_k r^{a_t}_{i_t s^\prime_k}$ and  $S(\hat{\mathbf{Q}}(t)):=\frac{1}{K}\sum_k \max_{a'}\hat{Q}_{s^\prime_k, a'}(t)$, respectively. For the purpose of analysis, we also tune the update slightly by substracting a small constant $(1-\gamma)\varepsilon/4$. Consequently, AsyncQVI is equivalent to
\begin{equation}\label{apx_Q}
Q_{i,a}(t+1)
=
\begin{cases}
r(t)
    +\gamma S(\hat{\mathbf{Q}}(t)) - (1-\gamma)\varepsilon/4 & \hbox{$t\in \mathscr{T}^{i,a}$;}\\
    Q_{i,a}(t), &\hbox{$t\notin \mathscr{T}^{i,a}$.}
    \end{cases}
\end{equation}

For memory efficiency, we do not form $\mathbf{Q}\in\mathbb{R}^{|\mathcal{S}||\mathcal{A}|}$. Instead, since only the values $\max_{a'} Q_{i,a'}$ are used for update, we maintain two vectors $\mathbf{v},\pi \in \mathbb{R}^{|\mathcal{S}|}$; at each iteration $t$, we ensure $v_i(t) = \max_a Q_{i,a}(t)$, $\pi_i(t) = \arg\max_a Q_{i,a}(t)$ and $\hat{v}_j(t)$ $= \max_{a^\prime} \hat{Q}_{j,a^\prime}(t)$.
By this means, we reduce the memory complexity from $\mathcal{O}(|\mathcal{S}||\mathcal{A}|)$ to $\mathcal{O}(|\mathcal{S}|)$, which is of a great advantage in real applications. 


%
%
%

\begin{remark}[Coordinate Selection] To guarantee convergence, the coordinate should be selected to satisfy Assumption~\ref{assump}.
In practice, however, if all agents have similar powers, one can simply apply either uniformly random or globally cyclic selections.
\end{remark}
\begin{remark}[Parallel Overhead]
In AsyncQVI, overhead can only occur during copying variable from global memory to the local memory (Line~\ref{alg_copy}) and where a memory lock is implemented (Lines~\ref{alg_if}-\ref{alg_unlock}). For the former case, the time complexity is $\mathcal{O}(|\mathcal{S}|)$, which is negligible when $\mathcal{O}(|\mathcal{S}|)$ is small or the process of querying samples is much slower. Otherwise, one can consider copying in a less frequent fashion, i.e., updating $\hat{\mathbf{v}}$ every $l_0$ iterations. 
Although it will increase $B_2$ by $l_0$, the sample complexity is still near-optimal as long as $B_1+B_2 = \mathcal{O}(|\mathcal{S}||\mathcal{A}|)$.
For the latter case, a memory lock (e.g. mutex) ensures that $v_i$ and $\pi_i$ are indeed the maximum value and a maximizer of the vector $\mathbf{Q}_i$, respectively. Since only two scalars are accessed and altered, the collision is rare.
\end{remark}


\subsection{Convergence Analysis}\label{sec_num_sta}
Next, we establish convergence for AsyncQVI; all proofs in this section are deferred to Appendix B. To distinguish different sequences, we let $(\mathbf{Q}^{\mathbb{E}}(t))$ denote the asynchronous coordinate update sequence generated through Eq.~\eqref{acu_Q}, where the superscript represents the updates with real expectations.
Specifically, if AsyncQVI produces a sequence according to Eq.~\eqref{apx_Q}
with $\hat{\mathbf{Q}}(t) =[ Q_{1,1}(\tau_{1,1}(t)),\dots,Q_{|\mathcal{S}|,|\mathcal{A}|}(\tau_{|\mathcal{S}|,|\mathcal{A}|}(t))]^\top$,
then
\begin{equation}\label{QE}
  Q^{\mathbb{E}}_{i,a}(t+1) =
\begin{cases}
    \bar{r}_i^a+\gamma \sum\limits_j p_{ij}^{a}
    \max\limits_{a^\prime}
    \hat{Q}^\mathbb{E}_{j,a^\prime}(t), & \hbox{$t\in \mathscr{T}^{i,a}$;} \\
    Q^{\mathbb{E}}_{i,a}(t), & \hbox{$t\notin \mathscr{T}^{i,a}$,}
\end{cases}
\end{equation}
where $\hat{\mathbf{Q}}^\mathbb{E}(t) =[ Q^{\mathbb{E}}_{1,1}(\tau_{1,1}(t)),\dots,Q^{\mathbb{E}}_{|\mathcal{S}|,|\mathcal{A}|}(\tau_{|\mathcal{S}|,|\mathcal{A}|}(t))]^\top$.
There are two things to notice:
\begin{enumerate}[(i)]
  \item $(\mathbf{Q}^{\mathbb{E}}(t))_{t=0}^{L}$ and $(\mathbf{Q}(t))_{t=0}^{L}$ have the same initial point;
  \item at any iteration, $(\mathbf{Q}^{\mathbb{E}}(t))_{t=0}^{L}$ shares exactly the same choice of coordinate $(i_t, a_t)$ and the same asynchronous delay with $(\mathbf{Q}(t))_{t=0}^{L}$.
\end{enumerate}
These properties are important to our analysis. Recall that we assume partial asynchronism (Assumption \ref{assump}) for AsyncQVI. Then Eq.~\eqref{QE} also meets Assumption \ref{assump}. Hence, Eq.~\eqref{QE} converges following the fact that $T$ is a $\gamma-$contraction and Proposition~\ref{Cov_Rate2}. Since Eq.~\eqref{apx_Q} is an approximation of Eq.~\eqref{QE}, we can leverage the convergence of Eq.~\eqref{QE} to establish the convergence of AsyncQVI. To this end, we first use Hoeffeding's Inequality~\citep{hoeffding1963probability} to analyze the sampling error. Specifically, if we take enough samples per iteration, then the error can be controlled with high probability.

\begin{proposition}[Sample Concentration]\label{SamCon}
With
$K = \big\lceil\frac{8}{(1-\gamma)^4\varepsilon^2}\log\big(\frac{4L}{\delta}\big)\big\rceil,$
AsyncQVI generates a sequence $(r(t), S(\hat{\mathbf{Q}}(t)))_{t=0}^{L-1}$ that satisfies
$\big|r(t) +\gamma S(\hat{\mathbf{Q}}(t))-\bar{r}_{i_t}^{a_t} - \gamma{\mathbf{p}_{i_t}^{a_t}}^\top\hat{\mathbf{v}}(t)\big|\leq\frac{(1-\gamma)\varepsilon}{4}$,
$\forall~ 0\leq t\leq L-1$, with probability at least $1-\delta$.
\end{proposition}

Proposition~\ref{SamCon} indeed establishes a control over a one-step approximation error between Eq.~\eqref{apx_Q} and Eq.~\eqref{QE} provided that $\hat{\mathbf{Q}}= \hat{\mathbf{Q}}^\mathbb{E}$. However, for the two sequences $(\mathbf{Q}(t))$ and $(\mathbf{Q}^{\mathbb{E}}(t))$ that only share the same initial point, the error can accumulate. To tackle this issue, we further utilize the $\gamma$-contraction property to weaken previously cumulative error. More specifically, if the newly made error and the previously accumulated error keep the ratio $(1-\gamma):1$ for each iteration, the overall error remains $(1-\gamma)\varepsilon+\gamma\varepsilon=\varepsilon$. By this means, we can control the difference between $(\mathbf{Q}(t))$ and $(\mathbf{Q}^{\mathbb{E}}(t))$ by induction.
\begin{proposition}\label{err1_ana}
Given the total iteration number $L$, accuracy parameters $\varepsilon$ and $\delta$, with $K = \big\lceil\frac{8}{(1-\gamma)^4\varepsilon^2}\log\big(\frac{4L}{\delta}\big)\big\rceil,$ AsyncQVI can generate a sequence $(\mathbf{Q}(t))_{t=1}^{L}$  satisfying
$\Vert\mathbf{Q}(t)-\mathbf{Q}^{\mathbb{E}}(t)\Vert_\infty
    \leq \varepsilon/2,~\forall ~1\leq t\leq L$
with probability at least $1-\delta$.
\end{proposition}
Since $({\mathbf{Q}}^{\mathbb{E}}(t))$ converges to $\mathbf{Q}^*$ linearly, combining Propositions~\ref{Cov_Rate2} and~\ref{err1_ana} gives the desired result.
\begin{theorem}[Linear Convergence]\label{err_Q}
Under Assumption~\ref{assump}, given accuracy parameters $\varepsilon$ and $\delta$, with
$L = \big\lceil2B_1+\frac{B_1+B_2-1}{1-\gamma}\log\big(\frac{2}{(1-\gamma)\varepsilon}\big)\big\rceil$ and
$K = \big\lceil\frac{8}{(1-\gamma)^4\varepsilon^2}\log\big(\frac{4L}{\delta}\big)\big\rceil,$ AsyncQVI can produce  $\mathbf{Q}(L)\in\mathbb{R}^{|\mathcal{S}||\mathcal{A}|}$ and $\mathbf{v}(L)\in\mathbb{R}^{|\mathcal{S}|}$ satisfying
$\Vert \mathbf{Q}^* - \mathbf{Q}(L) \Vert_\infty \leq \varepsilon$ and
$\Vert \mathbf{v}^* - \mathbf{v}(L) \Vert_\infty \leq \varepsilon$
with probability at least $1-\delta$.
\end{theorem}

\subsection{$\varepsilon$-optimal Policy}\label{sec_pol}
In the following theorem, we show that the vector $\pi$ maintained through the iterations is an $\varepsilon$-optimal policy; the proof is deferred to Appendix C. Using this theorem, we shall present the sample
complexity of AsyncQVI in Corollary~\ref{e_opt_policy}.

\begin{theorem}\label{e_policy}
Under Assumption~\ref{assump}, given accuracy parameters $\varepsilon$ and $\delta$, with
$L = \big\lceil2B_1+\frac{B_1+B_2-1}{1-\gamma}\log\big(\frac{2}{(1-\gamma)\varepsilon}\big)\big\rceil$ and
$K = \big\lceil\frac{8}{(1-\gamma)^4\varepsilon^2}\log\big(\frac{4L}{\delta}\big)\big\rceil,$ AsyncQVI returns an $\varepsilon$-optimal policy $\pi$ with probability at least $1-\delta$.
\end{theorem}



\begin{corollary}~\label{e_opt_policy}
Under Assumption~\ref{assump}, AsyncQVI returns an $\varepsilon$-optimal policy $\mathbf{\pi}$ with probability at least $1-\delta$ at the sample complexity
$$\tilde{\mathcal{O}}\big(\frac{B_1+B_2}{(1-\gamma)^5\varepsilon^2}\log(\frac{1}{\delta})\big).$$
\end{corollary}
Hence, if $B_1+B_2 = \mathcal{O}(|\mathcal{S}||\mathcal{A}|)$, then AsyncQVI has a near-optimal sample complexity.
%

Moreover, given the complete knowledge of transition $\mathrm{P}$ and reward $\mathrm{r}$, we can also solve it asynchronous parallelly. To utilize AsyncQVI, one can build a \emph{generative model} in $\tilde{O}(|\mathcal{S}|^2|\mathcal{A}|)$ prepossessing time~\citep{wang2017randomized}, and the \texttt{GM} produces a sample in $\tilde{O}(1)$ arithmetic operations. In this sense, AsyncQVI also has the following computational complexity results.
\begin{corollary}[Computational Complexity]\label{e_policy_corollary2}
Given a DMDP $(\mathcal{S}, \mathcal{A}, \mathrm{P}, \mathrm{r}, \gamma)$, under Assumption~\ref{assump} AsyncQVI returns an $\varepsilon$-optimal policy with probability at least $1-\delta$ at the computational complexity
$$\tilde{\mathcal{O}}\big(|\mathcal{S}|^2|\mathcal{A}|+\frac{|\mathcal{S}||\mathcal{A}|}{(1-\gamma)^5\varepsilon^2}\log(\frac{1}{\delta})\big),$$
provided that $B_1+B_2=\mathcal{O}(|\mathcal{S}||\mathcal{A}|)$.
\end{corollary}
\section{Numerical Experiments}
\subsection{Sailing Problem}
To investigate the performance of AsyncQVI, we solve the sailing problem from \citet{sailing} on a $100\times 100$ grid with 80000 states and 8 actions. Each state contains the sailor's current position $(x,y)$ and the wind direction. Each action is one of the eight directions
$\{(0,1),$ $(0,-1),$ $(1,0),$ $(-1,0),$ $(1,1),$ $(1,-1),$ $(-1,1),$ $(-1,-1)\}$.
The goal is to reach the target position $(50, 50)$ at the lowest cost. Different from the original settings, we add more randomness to the system. Under the action $(\delta_x, \delta_y) $, the sailor will be further affected by two drift noises: a mild wind noise $\mathcal{N}(0, \sigma_1^2)$ which occurs with probability (w.p.) 1 and a big vortex noise  $\mathcal{N}(0,\sigma_2^2)$ which occurs with a fairly small probability $p$. So, the next position is $(x+\delta_x+\mathcal{N}(0,\sigma_1^2), ~y+\delta_y+\mathcal{N}(0,\sigma_1^2))$ w.p. $1-p$ and $(x+\delta_x+\mathcal{N}(0,\sigma_1^2+\sigma_2^2), ~y+\delta_y+\mathcal{N}(0,\sigma_1^2+\sigma_2^2))$ w.p. $p$.
The wind direction at next time maintains its current direction w.p. 0.3, changes 45 degrees to either direction w.p. 0.2 each direction, changes 90 degrees to either direction w.p. 0.1 each, changes 135 degrees to either direction w.p. 0.04 each, and reverses direction w.p. 0.02. We set the instant reward as
$d \times \vert\frac{\text{angle between wind and action directions}}{45}\vert$,
where $d$ is a constant hyperparameter. When the reward is lower, we can take it as a higher cost. If the sailor reaches the target position, the reward is 1.

\subsection{Implementation}
We compare five algorithms with a sample oracle ($\mathcal{SO}$): AsyncQVI, Asynchronous-Parallel Q-learning with constant stepsize (AQLC), Asynchronous-Parallel Q-learning with diminishing stepsize (AQLD)\citep{Tsitsiklis1994},  Variance-reduced Value Iteration (VRVI)\citep{sidford2018variance}, and Variance-reduced Q-value Iteration (VRQVI)\citep{sidford2018near}. All algorithms and the $\mathcal{SO}$ are implemented in C++11.
We use the \texttt{thread} class and \texttt{pthread.h} for parallel computing.

The tests were performed with 20 threads running on two 2.5GHz 10-core Intel Xeon E5-2670v2 processors. We chose the optimal sample method (uniformly random, cyclic, Markovian sampling) and optimal hyperparameters (sample number, iteration number, learning rate, exploration rate) for each algorithm individually. The learning rate of AQLD was set as $1/t^{0.51}$ according to its  theoretical analysis, where $t$ is the iteration number. Our code is available in 
\href{https://github.com/uclaopt/AsyncQVI}{\texttt{https://github.com/uclaopt/AsyncQVI}}.

\begin{figure*}
    \centering
    \includegraphics[width=.825\textwidth]{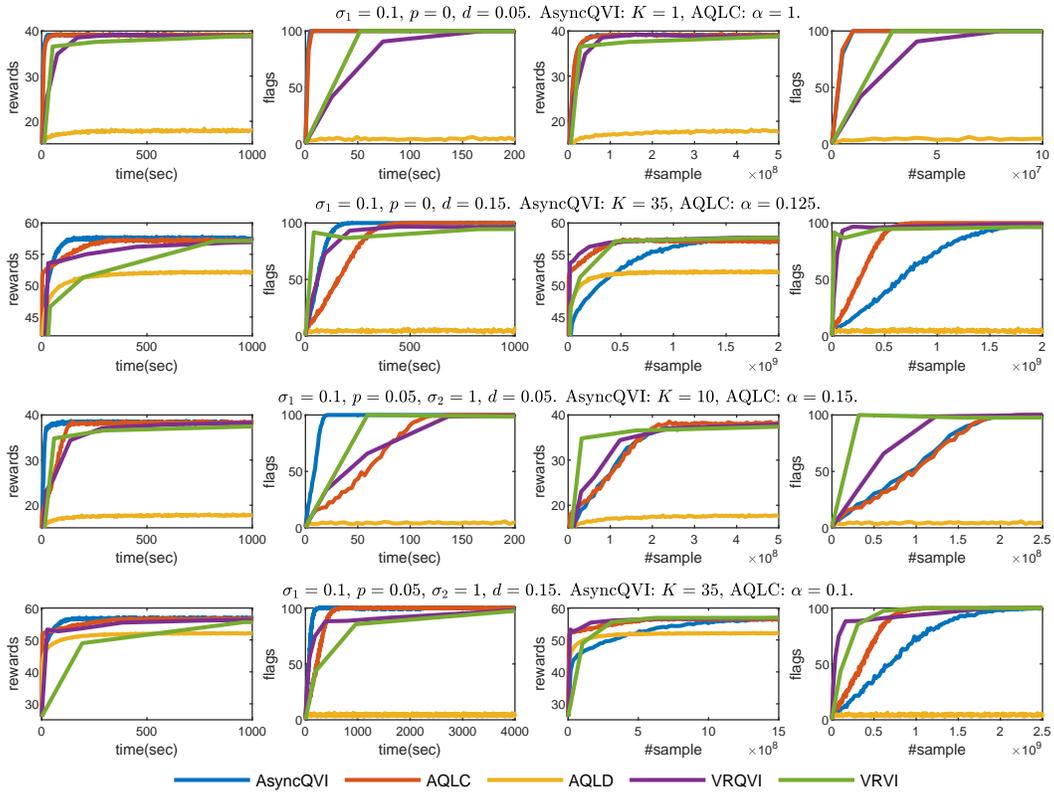}
    \vspace{2em}
    \caption{Performance Comparison Under Various Settings.}
    \label{fig1}
\end{figure*}
\begin{figure*}
    \centering
    \vspace{2em}
    \includegraphics[width=.825\textwidth]{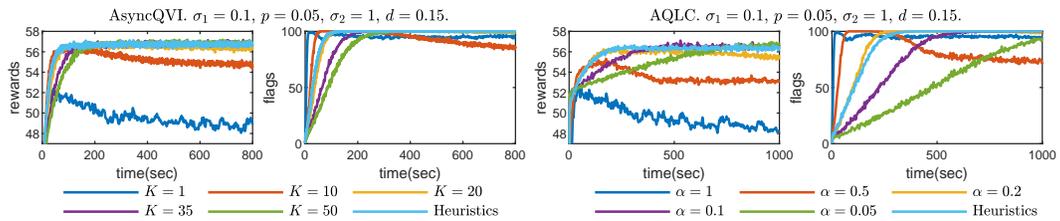}
    \vspace{2em}
    \caption{Performance Comparison With Different Sample Numbers Or Learning Rates.}
    \label{fig2}
\end{figure*}
\begin{figure*}
  \centering
  \vspace{2em}
  \includegraphics[width=.45\textwidth]{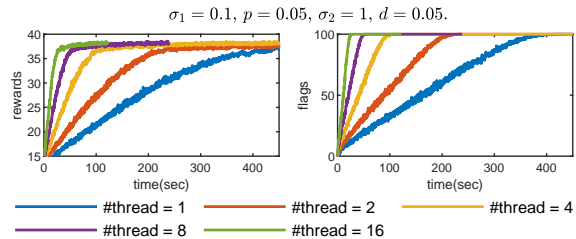}
  \vspace{2em}
  \caption{Parallel Performance Of AsyncQVI.}
  \label{parallel}
\end{figure*}

\subsection{Policy Evaluation}
Given a policy, we let the agent start from a random initial state and take actions following the policy for 200 steps. Then, we evaluate the policy by recording the total discounted rewards ($\gamma=0.99$) and whether the agent reaches the target position (flag $=1$ if so). We repeat 100 episodes of this process and calculate the average total discounted rewards and total flags. A policy with higher rewards and more flags is preferred.

We test with different randomness and rewards which represent various MDP settings (see Figure~\ref{fig1}). In the first test, one-step transition rewards are dominated by rewards for reaching the target  ($d=0.05$ is very small compared with 1) and only the wind noise is considered in positioning. The agent mainly aims at finding the target, which is relatively easy with minor noises. This leads to a fast convergence of policies with low sampling request and bold learning rate. In the second test, with increasing transition rewards ($d=0.15$), the agent needs to take a more economical way to reach the goal. This prolongs the learning process with more samples and more prudent learning rate. The next two tests make the situation more complicated with a big vortex noise, which gives rise to higher sampling numbers and more conservative learning rates. This phenomenon occurs in VRVI and VRQVI as well. We skip the detailed parameters here.

In these four tests, AsyncQVI and AQLC are almost equivalently outstanding in terms of time and achieve an at least $10\times$ speedup compared to VRQVI and VRVI with 20 threads running parallel. Further, VRQVI and VRVI have lower sample complexities, especially on complicated cases. The testing results verify our theory. In the sequel, we further analyze the performance of AsyncQVI and AQLC and provide heuristics on how to set sample number and learning rate.

\subsection{Performance Analysis and Heuristics}
Recall that AsyncQVI derives from the Q-value operator $T$ (see Eq.~\eqref{Q_value_operator}). Let $T_\alpha:= (1-\alpha)I + \alpha T$, where $\alpha$ is the learning rate. One can get AQLC through the same approach. What's special is, AQLC only takes one sample each time. This seems to be a very inaccurate approximation and might cause devastating error. However, note that when applying $T_\alpha$, sample range is also discounted by $\alpha$. For fixed $\delta$ and $\epsilon$, the requested sample number $m$ decreases quadratically with respect to $\alpha$, since $m\geq C\frac{\alpha^2}{\epsilon^2}\log\big(\frac{1}{\delta}\big)$ by Hoeffeding¡¯s Inequality~\citep{hoeffding1963probability}. Hence, when $\alpha$ is smaller, AQLC converges more stably. On the other hand, a tiny learning rate also leads to slow progress, since $T_\alpha$'s contractive factor $(1-\alpha+\alpha\gamma)$ approaches $1$. Similarly, for AsyncQVI, when the sample number $K$ is larger, it converges more stably but also more slowly. Therefore, we propose a trade-off heuristic of adaptively increasing the sample number or decreasing the learning rate. Specifically, in our test, we set $K_t=\min(\lfloor t^{0.175} \rfloor, 35)$ for AsyncQVI and $\alpha_t = \max(t^{-0.1}, 0.1)$ for AQLC, where $t$ is the iteration number. The results are depicted in Figure~\ref{fig2}.

The above interpretation also shows that AQLC is a special case of AsyncQVI (with $T_{\alpha}$ and $K=1$), which explains the similarity in their optimal performances. However, since AsyncQVI takes $\frac{1}{|\mathcal{A}|}\times$ memory of AQLC, our algorithm is still preferable for high dimensional applications.

\subsection{Parallel Performance}
We also test the parallel speedup performance of AsyncQVI using 1, 2, 4, 8, and 16 threads (see Figure \ref{parallel}). The result shows an ideal linear speedup.

\subsection{Summary}
AsyncQVI and AQLC have similar numerical performance, and they are faster than VRQVI, VRVI and AQLD. In general, async algorithms speed and scale up very well as the number of threads increases, and AsyncQVI is not an exception. On the other hand, AsyncQVI requires only $\mathcal{O}(|\mathcal{S}|)$ memory, which is much less than the $\mathcal{O}(|\mathcal{S}||\mathcal{A}|)$ memory of the other three; recall Table~\ref{asyn_compare}. Therefore, AsyncQVI can solve much larger problem instances.

\section{Conclusions and Future Work}\label{sec_conclusion}
In this paper, we propose an async-parallel algorithm AsyncQVI. Under mild asynchronism conditions, our algorithm achieves near-optimal sample complexity and minimal memory requirement. To the best of our knowledge, AsyncQVI is the first async-parallel algorithm for DMDPs with a generative model that has an explicit sample complexity.

For future work, we plan to integrate function approximation and policy exploration. Recently, a line of work established sample complexity results in this direction, e.g., \cite{yang2019sample, yang2019reinforcement, chen2018scalable}. We will also consider extending to continuous cases.



\subsubsection*{Acknowledgements}
This work of Yibo Zeng is partially supported by Fudan University Exchange Program and UCEAP Reciprocal Program for his
short-term visit to UCLA. The work of Fei Feng and Wotao Yin is partially supported by National Key R\&D Program of China 2017YFB02029 and NSF DMS-1720237.

\bibliography{AsyncQVI_AISTATS}

\newpage
\onecolumn
\begin{center}
    \textbf{\Large Supplementary Material}
\end{center}

\begin{appendices}
\section{Missing Proof of Proposition~\ref{Cov_Rate2}}

Recall that for each iteration $t$, we update the $i$th coordinate $x_i$ if $t\in\mathscr{T}^{i}$. Specifically,
\begin{equation}\label{update}
  x_{i}(t+1) =
  \begin{cases}
   G_i(\hat{\mathbf{x}}(t)),
                    & \hbox{$t\in \mathscr{T}^{i}$;} \\
      x_{i}(t),   & \hbox{$t\notin  \mathscr{T}^{i}$,}
  \end{cases}
\end{equation}
where $\hat{\mathbf{x}}(t):=[x_{1}(\tau_{1}(t)), \ldots,
      x_{n}(\tau_{n}(t))]^\top$.
For analysis, we sort $\mathscr{T}^i$ into a sequence $(t^i_k)_{k\geq 0}$, where $t_0^{i}$ is the first element of $\mathscr{T}^{i}$ and $t_k^{i}$ is the $(k+1)$th. Then Theorem~\ref{Cov_Rate} bounds $|x_{i}(t) - x^*_{i}|$ in a staircase decreasing way: $|x_{i}(t) - x^*_{i}|$
will contract when $t\in\mathscr{T}^{i}$, or equivalently, $t=t_k^i$ for some $k$.

\begin{theorem}[Staircase Decreasing]\label{Cov_Rate}
  Consider the iteration~\eqref{update} under Assumption~\ref{assump}. Suppose that $G$ is $\gamma$-contractive under infinity norm and $\mathbf{x}^*$ is the fixed-point of $G$. For each $t\geq B_1$ and $i\in\{1, 2, \cdots, n\}$, if $t\in(t_k^i, t_{k+1}^i]$ for some $k$, then $x_i(t)$ satisfies
  \begin{equation}\label{CR1}
    |x_{i}(t) - x^*_{i}| \leq \Vert \mathbf{x}(0) - \mathbf{x}^* \Vert_\infty \rho^{t_k^i-B_1},
  \end{equation}
  where $\rho := \gamma^{\frac{1}{B_1+B_2-1}}$.
\end{theorem}
\begin{proof}
We first claim that for each $t\geq B_1$ and $i\in\{1, 2, \cdots, n\}$, there exists some $k\geq 0$ such that $t\in(t_k^i, t_{k+1}^i]$. This follows from Assumption~\ref{assump} (a), where $t_0^i\leq B_1-1$, $\forall~i$.

Now we prove Eq.~\eqref{CR1} by induction. One could check
\begin{equation*}
\Vert \mathbf{x}(t) - \mathbf{x}^*\Vert_\infty
  \leq \Vert \mathbf{x}(0) - \mathbf{x}^* \Vert_\infty,~\forall~t \geq 0
\end{equation*}
as a corollary of~\citep[Theorem 2]{feyzmahdavian2014convergence} or by another induction. We skip the details here. Thus for the basic case,
\begin{equation*}
        \max_{0\leq t\leq B_1}
        \big\{\Vert \mathbf{x}(t) - \mathbf{x}^*\Vert_\infty\rho^{-t}\big\}
\leq    \max_{0\leq t\leq B_1}
        \big\{\Vert \mathbf{x}(0) - \mathbf{x}^*\Vert_\infty\rho^{-t}\big\}
\leq    \Vert \mathbf{x}(0) - \mathbf{x}^* \Vert_\infty \rho^{-B_1},
\end{equation*}
which gives that for each $t\leq B_1$ and $i\in\{1, 2, \cdots, n\}$,
\begin{equation*}
  |x_{i}(t) - x^*_{i}|
  \leq \Vert \mathbf{x}(0) - \mathbf{x}^* \Vert_\infty \rho^{t-B_1}.
\end{equation*}
Since $\rho^t$ is decreasing, we can further obtain that
\begin{equation*}
|x_{i}(B_1) - x^*_{i}|
 \leq \Vert \mathbf{x}(0) - \mathbf{x}^* \Vert_\infty \rho^{t_k^i-B_1},
\end{equation*}
if $B_1\in(t_k^i, t_{k+1}^i]$ for some $k$.

For the induction step, we assume that Eq.~\eqref{CR1} holds for all $t\geq B_1$ up to some $t^\prime$. For a fixed $i\in\{1, 2,\cdots, n\}$, supposing that $t^\prime\in(t_{k^\prime}^i, t_{k^\prime+1}^i]$ for some $k^\prime$, then we analyze the scenario at $(t^\prime+1)$ as two cases.

\noindent \textbf{Case 1:} $t^\prime \notin \mathscr{T}^{i}$, i.e., we do not update coordinate $i$ at iteration $t^\prime$. Hence, $x_{i}(t^\prime + 1) = x_{i}(t^\prime)$ and $t^\prime +1 \in (t^{i}_{k^\prime}, t^{i}_{k^\prime + 1}]$. Then Eq.~\eqref{CR1}  follows directly.

\noindent \textbf{Case 2:}  $t^\prime \in \mathscr{T}^{i}$, i.e., the $i$th coordinate is updated at iteration $t^\prime$ and $t^\prime = t^{i}_{k^\prime + 1}$. Since $G$ is $\gamma$-contractive under infinity norm, we have
\begin{equation}\label{Cov_Rate_Eq4}
\begin{aligned}
      |x_{i}(t^\prime+1) - x_{i}^*|
=     |G_{i}(\hat{\mathbf{x}}(t)) - x^*_{i}|
\leq&~\Vert G(\hat{\mathbf{x}}(t)) - \mathbf{x}^* \Vert_\infty\\
\leq&~\gamma \max_{j}
      \big|x_{j}(\tau_{j}(t^\prime))- x^*_{j}\big|.
\end{aligned}
\end{equation}
For a fixed $j\in\{1, 2,\cdots, n\}$, suppose that $\tau_{j}(t^\prime)
  \in (t^{j}_{k_\tau}, t^{j}_{k_\tau+1}]$ for some $k_\tau$.
Then the induction hypothesis gives
$ \big|x_{j}(\tau_{j}(t^\prime)) - x_{j}^*\big|
  \leq \Vert \mathbf{x}(0) - \mathbf{x}^* \Vert_\infty
        \rho^{t_{k_\tau}^{j}-B_1}.$
Since $\tau_{j}(t^\prime)
\leq t_{k_\tau}^{j}+B_1$ by Assumption~\ref{assump} (a) and $\tau_{j}(t^\prime) \geq t^\prime-B_2+1$ by Assumption~\ref{assump} (b), we obtain
\begin{align}
        \gamma \big|x_{j}(\tau_{j}(t^\prime)) - x_{j}^*\big|
  \leq&~\gamma \Vert \mathbf{x}(0) - \mathbf{x}^* \Vert_\infty
        \rho^{t_{k_\tau}^{j}-B_1}\nonumber\\
  \leq&~\gamma\Vert \mathbf{x}(0) - \mathbf{x}^* \Vert_\infty
        \rho^{\tau_{j}(t^\prime) - 2B_1}\nonumber\\
  \leq&~\gamma\Vert \mathbf{x}(0) - \mathbf{x}^* \Vert_\infty
        \rho^{t^\prime-2B_1-B_2-1}\nonumber\\
  =&~   \Vert \mathbf{x}(0) - \mathbf{x}^* \Vert_\infty
        \rho^{t^{i}_{k^\prime + 1}-B_1},\label{Cov_Rate_Eq5}
\end{align}
where the equality holds since $\gamma = \rho^{B_1+B_2-1}$ by definition and $t^\prime = t^{i}_{k^\prime + 1}$. Notice that $t^\prime + 1\in (t^{i}_{k^\prime + 1}, t^{i}_{k^\prime + 2}]$. Inserting Eq.~\eqref{Cov_Rate_Eq5} back into Eq.~\eqref{Cov_Rate_Eq4} yields the desired result.
\end{proof}

One may note that if $t\in(t_k^i, t_{k+1}^i]$, then
$t_k^i + B_1  \geq t$ by Assumption~\ref{assump} (a).
Hence, Proposition~\ref{Cov_Rate2} is a direct consequence of Theorem~\ref{Cov_Rate}.

\section{Missing Proof from Section 4.1}
To analyze the sampling error, we first review Hoeffeding's Inequality~\citep{hoeffding1963probability}.

\begin{lemma}[Hoeffeding's Inequality \citep{hoeffding1963probability}]\label{Hof}
Let $X_1,\cdots,X_m$
be i.i.d real valued random variables with $X_j\in[a_j,b_j]$ and  $Y=\frac{1}{m}\sum_{j=1}^m X_j$. For all $\varepsilon\geq 0$,
\begin{equation*}
\mathbb{P}\big[\big|Y-\mathbb{E}[Y]\big|\geq \varepsilon\big]
\leq2e^{\frac{-2m^2\varepsilon^2}{\sum_{j=1}^m(b_j-a_j)^2}}.
\end{equation*}
\end{lemma}
By Hoeffeding's Inequality, the error between the sample averages and the true expectations can be controlled with enough number of samples. Specifically, we have:
\begin{lemma}\label{SamCon_lem}
Given a constant $L$, with
$K = \big\lceil\frac{8}{(1-\gamma)^4\varepsilon^2}\log(\frac{4L}{\delta})\big\rceil$
samples, AsyncQVI returns $r(t)$ and $S(\hat{\mathbf{Q}}(t))$ satisfying
\begin{equation*}
\vert r(t) - \bar{r}_{i_t}^{a_t}\vert \leq \frac{(1-\gamma)^2\varepsilon}{4}, \qquad
  \vert S(\hat{\mathbf{Q}}(t))- {\mathbf{p}_{i_t}^{a_t}}^\top\hat{\mathbf{v}}(t)\vert
  \leq\frac{(1-\gamma)\varepsilon}{4}
\end{equation*}
with probability at least $1-\frac{\delta}{L}$.
\end{lemma}

\begin{proof}
As we explained before, both $r(t)$ and $S(\hat{\mathbf{Q}}(t))$ are averages of $K$ i.i.d. samples with $\mathbb{E}[r(t)] =\sum_j p_{i_tj}^{a_t} r_{i_tj}^{a_t} := \bar{r}_{i_t}^{a_t}$ and $\mathbb{E}[S(\hat{\mathbf{Q}}(t))]= \sum_j p_{i_tj}^{a_t}\hat{v}_j(t) := {\mathbf{p}_{i_t}^{a_t}}^\top\hat{\mathbf{v}}(t)$.
Since we assume $r_{ij}^a\in[0,1]$, it is easy to verify $0 \leq \hat{\mathbf{v}}(t) \leq \frac{1}{1-\gamma}$ by induction. We skip the details here. Then letting $K = \big\lceil\frac{8}{(1-\gamma)^4\varepsilon^2}\log\big(\frac{4L}{\delta}\big)\big\rceil$, we can obtain that
\begin{equation*}
\begin{aligned}
    &  \mathbb{P}\Big[\vert r(t) - \bar{r}_{i_t}^{a_t}\vert
                       \geq\dfrac{(1-\gamma)^2\varepsilon}{4}\Big]
\leq   2 e^{ \frac{-2K^2(1-\gamma)^4\varepsilon^2}
        {16K} }
\leq    \dfrac{\delta}{2L};\\
    &  \mathbb{P}\Big[\vert S(\hat{\mathbf{Q}}(t)) - {\mathbf{p}_{i_t}^{a_t}}^\top\hat{\mathbf{v}}(t)\vert
                       \geq\dfrac{(1-\gamma)\varepsilon}{4}\Big]
\leq   2 e^{ \frac{-2K^2(1-\gamma)^4\varepsilon^2}
        {16K} }
\leq    \dfrac{\delta}{2L}.
\end{aligned}
\end{equation*}
\end{proof}

\begin{proof}[Proof of Proposition~\ref{SamCon}]
For a fixed iteration $t$, by Lemma~\ref{SamCon_lem},
\begin{equation*}
\begin{aligned}
        \big|r(t) +\gamma S(\hat{\mathbf{Q}}(t))-\bar{r}_{i_t}^{a_t} -
        \gamma{\mathbf{p}_{i_t}^{a_t}}^\top\hat{\mathbf{v}}(t)\big|
  \leq  \vert r(t)- \bar{r}_{i_t}^{a_t}\vert
        + \gamma \vert S(\hat{\mathbf{Q}}(t)) -
        {\mathbf{p}_{i_t}^{a_t}}^\top\hat{\mathbf{v}}(t)\vert
  \leq  \frac{(1-\gamma)\varepsilon}{4}
\end{aligned}
\end{equation*}
holds with probability at least $1-\frac{\delta}{L}$. Taking a union bound over all $0\leq t\leq L-1$ iterations gives the desired result.
\end{proof}

\begin{proof}[Proof of Proposition~\ref{err1_ana}]
We denote by $\mathcal{E}_1$ the event
\begin{equation*}
\Big\{\big|r(t) +\gamma S(\hat{\mathbf{Q}}(t)) -\bar{r}_{i_t}^{a_t} - \gamma{\mathbf{p}_{i_t}^{a_t}}^\top\hat{\mathbf{v}}(t)\big|\leq\frac{(1-\gamma)\varepsilon}{4},~\forall~0\leq t\leq L-1\Big\}.
\end{equation*}
By Proposition~\ref{SamCon}, $\mathcal{E}_1$ occurs with probability at least $1-\delta$. Next, we condition on $\mathcal{E}_1$ and prove
\begin{equation}\label{eq:QQE}
    \Vert\mathbf{Q}(t)-\mathbf{Q}^{\mathbb{E}}(t)\Vert_\infty
    \leq \frac{\varepsilon}{2},~\forall ~1\leq t\leq L
\end{equation}
by induction. The basic case is trivial.
For the induction step, we analyze the scenario at $t+1$ as two cases. When $t\notin \mathscr{T}^{i,a}$, $|Q_{i, a}(t+1)- Q^{\mathbb{E}}_{i, a}(t+1)|\leq\varepsilon/2$ follows from the hypothesis, since Eqs.~\eqref{apx_Q} and~\eqref{QE} give that
\begin{equation*}
  Q_{i, a}(t+1)- Q^{\mathbb{E}}_{i, a}(t+1)
= Q_{i,a}(t) - Q^{\mathbb{E}}_{i, a}(t).
\end{equation*}
When $t \in \mathscr{T}^{i,a}$, by Eq.~\eqref{apx_Q}, Eq.~\eqref{QE} and triangle inequality, we have that 
\begin{equation*}
\begin{aligned}
 &~     \big|Q_{i, a}(t+1)-Q^{\mathbb{E}}_{i, a}(t+1)\big|\\
=&~     \Big|r(t)+\gamma S(\hat{\mathbf{Q}}(t))  - \dfrac{(1-\gamma)\varepsilon}{4}
        -\bar{r}_i^a   - \gamma \sum_j p_{ij}^{a}
        \max_{a^\prime}
        \hat{Q}^\mathbb{E}_{j,a^\prime}(t)\Big|\\
\leq&~ \Big| r(t) + \gamma S(\hat{\mathbf{Q}}(t))- \bar{r}_i^a
        - \gamma {\mathbf{p}_{i}^{a}}^\top\hat{\mathbf{v}}(t)
        - \dfrac{(1-\gamma)\varepsilon}{4}\Big|
        +\Big|{\gamma\mathbf{p}_{i}^{a}}^\top\hat{\mathbf{v}}(t)
        -\gamma\sum_j p_{ij}^{a}
        \max_{a^\prime} \hat{Q}^\mathbb{E}_{j,a^\prime}(t) \Big\vert\\
\leq&~  \Big| r(t) + \gamma S(\hat{\mathbf{Q}}(t))- \bar{r}_i^a
        - \gamma {\mathbf{p}_{i}^{a}}^\top\hat{\mathbf{v}}(t)\Big|
        + \dfrac{(1-\gamma)\varepsilon}{4}
        + \gamma\sum_j p_{ij}^{a}
        \big\vert \max_{a^\prime} \hat{Q}_{j,a^\prime}(t)
        - \max_{a^\prime} \hat{Q}^\mathbb{E}_{j,a^\prime}(t) \big\vert.
\end{aligned}
\end{equation*}
By definition of $\mathcal{E}_1$ and the induction hypothesis, we further obtain that
\begin{equation*}
        |Q_{i, a}(t+1)- Q^{\mathbb{E}}_{i, a}(t+1)|
\leq    \dfrac{(1-\gamma)\varepsilon}{4} + \dfrac{(1-\gamma)\varepsilon}{4}
        + \gamma\dfrac{\varepsilon}{2}
       = \dfrac{\varepsilon}{2},
\end{equation*}
which completes the proof.
\end{proof}

\begin{proof}[Proof of Theorem~\ref{err_Q}]
By Proposition~\ref{Cov_Rate2},
\begin{equation*}\label{err_Q_eq1}
        \Vert\mathbf{Q}^* - \mathbf{Q}^{\mathbb{E}}(L)\Vert_\infty
\leq    (1-\gamma)^{-1} \rho^{L-2B_1}
=       (1-\gamma)^{-1} \gamma^{\frac{L-2B_1}{B_1+B_2-1}}.
\end{equation*}
Notice that $\gamma = (1-(1-\gamma))\leq e^{-(1-\gamma)}$. We have that
\begin{equation}~\label{err_Q_eq2}
        \Vert\mathbf{Q}^* - \mathbf{Q}^{\mathbb{E}}(L)\Vert_\infty
\leq    (1-\gamma)^{-1}
        e^{-(1-\gamma)\frac{L-2B_1}{B_1+B_2-1}}
\leq    \frac{\varepsilon}{2},
\end{equation}
where the last inequality holds with $L = \big\lceil2B_1+\frac{B_1+B_2-1}{1-\gamma}
\log\big(\frac{2}{(1-\gamma)\varepsilon}\big)\big\rceil $.
Then, by Proposition~\ref{err1_ana}, with probability at least $1-\delta$,
\begin{equation}~\label{err_Q_eq3}
    \Vert\mathbf{Q}^{\mathbb{E}}(L)-\mathbf{Q}(L)\Vert_\infty
    \leq \frac{\varepsilon}{2}.
\end{equation}
Inserting Eq.~\eqref{err_Q_eq3} back into Eq.~\eqref{err_Q_eq2} gives the desired result
\begin{equation*}
        \Vert \mathbf{Q}^* - \mathbf{Q}(L) \Vert_\infty
\leq    \Vert\mathbf{Q}^* - \mathbf{Q}^{\mathbb{E}}(L)  \Vert_\infty
        + \Vert\mathbf{Q}^{\mathbb{E}}(L) - \mathbf{Q}(L)\Vert_\infty
\leq    \varepsilon.
\end{equation*}
Then one can check $\Vert \mathbf{v}^* - \mathbf{v}(L) \Vert_\infty\leq \varepsilon$ at ease.
\end{proof}

\section{Missing Proof of Theorem~\ref{e_policy}}
After $L$ iterations, AsyncQVI returns a policy $\pi(L)$ with $\pi_i(L) = \arg\max_{a\in\mathcal{A}}Q_{i,a}(L)$. To show that $\pi(L)$ is $\varepsilon$-optimal, we first define a policy operator.



\begin{definition}[Policy Operator]
Given a policy $\pi$ and a vector $\mathbf{v}\in \mathbb{R}^{|\mathcal{S}|}$, the policy operator $T_\pi$: $\mathbb{R}^{|\mathcal{S}|}\rightarrow\mathbb{R}^{|\mathcal{S}|}$ is defined as
\begin{equation}~\label{policy_operator}
    [T_\pi\mathbf{v}]_{i}
=   \bar{r}_{i}^{\pi_i}
    +\gamma {\mathbf{p}_{i}^{\pi_i}}^\top \mathbf{v}
=   r_{i}^{\pi_i}
    +\gamma \sum_{j\in\mathcal{S}} p_{ij}^{\pi_i} v_j.
\end{equation}
\end{definition}
\begin{proposition}[$T_{\pi}$'s Properties]~\label{Tpi_property}
Given a policy $\pi$, for any vectors $\mathbf{v}$, $\mathbf{v^\prime}\in\mathbb{R}^{\mathcal{S}}$,
\begin{enumerate}[(a)]
  \item \textbf{Monotonicity: } if $\mathbf{v}\leq\mathbf{v}^\prime$, then $T_{\pi}\mathbf{v}\leq T_{\pi}\mathbf{v}^\prime$.
  \item \textbf{$\gamma$-Contraction:}
      $\Vert T _{\pi}\mathbf{v} - T_{\pi}\mathbf{v}^\prime\Vert_\infty
      \leq \gamma\Vert\mathbf{v}-\mathbf{v}^\prime\Vert_\infty$.
    \item $\mathbf{v}^{\pi}$ is the fixed-point of $T_{\pi}$.
\end{enumerate}
\end{proposition}
The proof is straightforward following the definition. We skip the details here.
\begin{lemma}\label{monotone}\emph{\citep{sidford2018variance}} Given a policy $\pi$, for any vector $\mathbf{v}\in\mathbb{R}^{\mathcal{S}}$, if there exists a $\mathbf{v^\prime}\in\mathbb{R}^{\mathcal{S}}$ such that $\mathbf{v}^\prime \leq \mathbf{v}$ and  $\mathbf{v}\leq T_{\pi}\mathbf{v}^\prime$, then
$\mathbf{v}\leq\mathbf{v}^\pi.$
\end{lemma}
\begin{proof}
By Proposition~\ref{Tpi_property} (a) and $\mathbf{v}^\prime\leq\mathbf{v}$, we first have $T_{\pi}\mathbf{v}^\prime \leq T_{\pi}\mathbf{v}$. Combining with $\mathbf{v}\leq T_{\pi}\mathbf{v}'$, we further obtain
$\mathbf{v}\leq T_{\pi}\mathbf{v}$.
By induction, one can check $\mathbf{v}\leq T^n_{\pi}\mathbf{v}$, $\forall~n\in\mathbb{N}$.
Moreover, since $T_\pi$ is a $\gamma$-contraction,
$\mathbf{v}^\pi = \lim_{n\rightarrow \infty} T_\pi^n \mathbf{v}$.
Hence, $\mathbf{v}\leq\lim_{n\rightarrow \infty} T_\pi^n \mathbf{v}=\mathbf{v}^\pi$.
\end{proof}

Next, we consider the special case that $\mathbf{v}(L)$ and $\pi(L)$ are both derived from AsyncQVI with
\begin{equation*}
   \pi_i(L) = \arg\max_{a} Q_{i,a}(L),\quad v_{i}(L) = \max_{a} Q_{i,a}(L), ~\forall i\in \mathcal{S}.
\end{equation*}
If $\|\mathbf{v}^* - \mathbf{v}^{\pi}\|_{\infty}\leq \varepsilon$, then $\pi$ is $\varepsilon-$optimal. To achieve this, we first show that $\mathbf{v}(L)$ satisfies Lemma \ref{monotone} (see Lemma \ref{AsyncQVI_monotone}). Then with Theorem~\ref{err_Q}, $\|\mathbf{v}^*-\mathbf{v}^{\pi}\|_{\infty}\leq\|\mathbf{v}^*-\mathbf{v}(L)\|_{\infty}\leq \varepsilon$.


\begin{lemma}\label{AsyncQVI_monotone}
Under Assumption~\ref{assump}, AsyncQVI generates a sequence of $\{\mathbf{v}(t)\}_{t=1}^{L}$ and $\{\pi(t)\}_{t=1}^{L}$ satisfying
\begin{equation}\label{e_optimal_policy_eq1}
  \mathbf{v}(t-1) \leq \mathbf{v}(t) \leq T_{\pi(t)}\mathbf{v}(t-1),
  ~\forall~1\leq t\leq L
\end{equation}
with probability at least $1-\delta$.
\end{lemma}
\begin{proof}
By Proposition~\ref{SamCon},
\begin{equation*}
  \big\vert r(t) + \gamma S(\hat{\mathbf{Q}}(t))- \bar{r}_{i_t}^{a_t} - \gamma {\mathbf{p}_{i_t}^{a_t}}^\top\hat{\mathbf{v}}(t)\big\vert
  \leq\frac{(1-\gamma)\varepsilon}{4},~\forall~0\leq t\leq L-1
\end{equation*}
holds with probability at least $1-\delta$.
Denote by $\mathcal{E}_2$ the event
\begin{equation*}
\Big\{
  r(t) + \gamma S(\hat{\mathbf{Q}}(t))-\frac{(1-\gamma)\varepsilon}{4}\leq
  \bar{r}_{i_t}^{a_t} + \gamma{\mathbf{p}_{i_t}^{a_t}}^\top \hat{\mathbf{v}}(t),~\forall~0\leq t\leq L-1\Big\}.
\end{equation*}
Then $\mathcal{E}_2$ occurs with probability at least $1-\delta$.

Now we condition on $\mathcal{E}_2$ and prove Eq.~\eqref{e_optimal_policy_eq1} by induction. For simplicity, we let $\mathbf{v}(-1)=\mathbf{v}(0)=\mathbf{0}$ and start our proof from $t=0$. Then the basic case holds. For the induction step, suppose that Eq.~\eqref{e_optimal_policy_eq1} is true for all $t$ up to some $t^\prime$.
Recall that in AsyncQVI, for each iteration, whether $v_i$ or $\pi_i$ will be updates depends on the value of $Q_{i,a}$. We hence analyze the scenario at $(t'+1)$ as two cases.

\noindent \textbf{Case 1:}
  $Q_{i_{t^\prime},a_{t^\prime}}(t^\prime+1) \leq v_{i_{t^\prime}}(t^\prime)$. Then $\mathbf{v}$ and $\pi$ will not be updated, i.e., $\mathbf{v}(t^\prime+1) = \mathbf{v}(t^\prime)$ and  $\pi(t^\prime+1)= \pi(t^\prime)$. In this case, the inequality $\mathbf{v}(t^\prime) \leq \mathbf{v}(t^\prime+1)$ follows directly. For the other part, by induction hypothesis we have
  \begin{equation*}
    \mathbf{v}(t^\prime+1)
    =       \mathbf{v}(t^\prime)
    \leq    T_{\pi(t^\prime)}\mathbf{v}(t^\prime-1)
    =      T_{\pi(t^\prime+1)}\mathbf{v}(t^\prime-1)
    \leq    T_{\pi(t^\prime+1)}\mathbf{v}(t^\prime),
  \end{equation*}
where the last inequality comes from $\mathbf{v}(t^\prime-1)\leq \mathbf{v}(t^\prime)$ and the monotonicity of $T_{\pi(t^\prime+1)}$.

\noindent \textbf{Case 2:}
  $Q_{i_{t^\prime},a_{t^\prime}}(t^\prime+1)> v_{i_{t^\prime}}(t^\prime)$. Then $\forall~i\in\mathcal{S}$,
\begin{description}
  \item[\quad Case 2.1: ]
  $i\neq i_{t^\prime}$. In this case, $v_{i}(t^\prime+1) = v_i(t^\prime)$ and $\pi_{i}(t^\prime+1)= \pi_{i}(t^\prime)$. Hence, once again by induction hypothesis and $T_\pi$'s monotonicity, we obtain
  \begin{equation*}
            v_{i}(t^\prime+1)
    =       v_{i}(t^\prime)
    \leq    \big[T_{\pi(t^\prime)}\mathbf{v}(t^\prime-1)\big]_{i}
    =       \big[T_{\pi(t^\prime+1)}\mathbf{v}(t^\prime-1)\big]_{i}
    \leq    \big[T_{\pi(t^\prime+1)}\mathbf{v}(t^\prime)\big]_{i}.
  \end{equation*}
  \item[\quad Case 2.2: ]
  $i=i_{t^\prime}$. According to Lines~\ref{alg_if} and~\ref{alg_pi} of Algorithm~\ref{AsyncQVI}, the $i$th coordinate of $\mathbf{v}$ is updated at iteration $t^\prime$ and the former inequality follows directly. For the latter inequality, by Line~\ref{alg_q} of Algorithm~\ref{AsyncQVI} we have
    \begin{equation*}
            v_{i}(t^\prime+1)
    =       Q_{i,a_{t^\prime}}(t^\prime+1)
    =       r(t^\prime) + \gamma S(\hat{\mathbf{Q}}(t^\prime))
            - \frac{(1-\gamma)\varepsilon}{4}.
    \end{equation*}
  By definition of $\mathcal{E}_2$ and $\pi_i(t^\prime+1) = a_{t^\prime}$, we obtain
  \begin{equation*}
            v_{i}(t^\prime+1)
    \leq    \bar{r}_{i}^{a_{t^\prime}}
            + \gamma {\mathbf{p}_{i}^{a_{t^\prime}}}^\top\hat{\mathbf{v}}(t^\prime)
    =       \big[T_{\pi(t^\prime+1)}
            \hat{\mathbf{v}}(t^\prime)\big]_{i}.
  \end{equation*}
Owing to $\hat{\mathbf{v}}(t^\prime)\leq \mathbf{v}(t^\prime)$ by induction hypothesis and the monotonicity of $T_{\pi(t^\prime+1)}$, we can complete our proof by
  \begin{equation*}
            v_{i}(t^\prime+1)
    \leq    \big[T_{\pi(t^\prime+1)}
            \hat{\mathbf{v}}(t^\prime)\big]_{i}
    \leq    \big[T_{\pi(t^\prime+1)}{\mathbf{v}}(t^\prime)\big]_{i}.
  \end{equation*}
\end{description}
\end{proof}

Finally, combining the results of Lemma \ref{monotone}, Lemma~\ref{AsyncQVI_monotone} and Theorem~\ref{err_Q}, we can establish Theorem~\ref{e_policy} at ease.
\end{appendices}

\end{document}